\documentclass[a4paper, 10pt, twoside]{article}

\usepackage{amsmath, amscd, amsfonts, amssymb, amsthm,
latexsym, url, color, todonotes, bm, framed, rotating,
enumerate} 
\usepackage{graphicx}
\usepackage[left=1in, right=1in, top=1.2in, bottom=1in, includefoot, headheight=13.6pt]{geometry}
\usepackage{mathtools}
\input{xy}
\xyoption{all}

\usepackage[colorlinks,breaklinks=true]{hyperref}
\usepackage[figure,table]{hypcap}
\hypersetup{
	bookmarksnumbered,
	pdfstartview={FitH},
	citecolor={black},
	linkcolor={black},
	urlcolor={black},
	pdfpagemode={UseOutlines}
}
\makeatletter
\newcommand\org@hypertarget{}
\let\org@hypertarget\hypertarget
\renewcommand\hypertarget[2]{%
  \Hy@raisedlink{\org@hypertarget{#1}{}}#2%
}
\makeatother

\newtheorem{theorem}{Theorem}[section]
\newtheorem{lemma}[theorem]{Lemma}
\newtheorem{corollary}[theorem]{Corollary}
\newtheorem{proposition}[theorem]{Proposition}

\theoremstyle{definition}
\newtheorem{definition}[theorem]{Definition}
\newtheorem{definition/proposition}[theorem]{Definition/Proposition}
\newtheorem{remark}[theorem]{Remark}

\newtheorem{conjecture}[theorem]{Conjecture}
\newtheorem{question}[theorem]{Question}

\newcommand{\xysquare}[8]{
\[\xymatrix{
#1 \ar@{#5}[r] \ar@{#6}[d] & #2 \ar@{#7}[d]\\
#3 \ar@{#8}[r] & #4
}\]
}

\newcommand{\bb}{\mathbb}

\renewcommand{\phi}{\varphi}

\newcommand{\roi}{\mathcal{O}}

\newcommand{\xto}{\xrightarrow}

\renewcommand{\cal}{\mathcal}
\renewcommand{\hat}{\widehat}
\renewcommand{\frak}{\mathfrak}
\newcommand{\indlim}{\varinjlim}

\renewcommand{\ker}{\operatorname{Ker}}

\DeclareMathOperator{\codim}{codim}

\DeclareMathOperator{\dlog}{dlog}

\DeclareMathOperator{\Gal}{Gal}

\DeclareMathOperator{\Hom}{Hom}

\DeclareMathOperator{\Spec}{Spec}

\DeclareMathOperator{\Val}{Val}

\newcommand{\CH}{C\!H}

\def\T{{\frak T}}
\def\et{{\mathrm{\acute{e}t}}}
\def\t{{\mathrm{t}}}
\def\Nis{{\mathrm{Nis}}}
\def\Cat{{\mathrm{Cat}}}

\newcommand{\rcoeq}[3]{\xymatrix{ #1\ar@/^3mm/[r]^f \ar@/_3mm/[r]_g & #2 \ar[l]_e\ar[r] & #3}}

\usepackage[bbgreekl]{mathbbol}
\DeclareSymbolFontAlphabet{\mathbbm}{bbold}

\usepackage{fancyhdr}

\usepackage{sectsty}
\sectionfont{\Large\sc\centering}
\chapterfont{\large\sc\centering}
\chaptertitlefont{\LARGE\centering}
\partfont{\centering}

\begin{document}
\itemsep0pt

\title{$p$-adic tame Tate twists}

\author{Morten L\"uders}

\date{}

\maketitle

\begin{abstract} 
Recently, Hübner-Schmidt defined the tame site of a scheme. We define $p$-adic tame Tate twists in the tame topology and prove some first properties. We establish a framework analogous to the Beilinson-Lichtenbaum conjectures in the tame topology for $p$-adic tame Tate twists and tame logarithmic deRham-Witt sheaves. Both only differ from their \'etale counterpart in cohomological degrees above the weight. These cohomology groups can be analysed using the Gersten conjecture which, at least conjecturally, has a nice shape in the tame topology. We prove the Gersten conjecture for tame logarithmic deRham-Witt sheaves for curves in positive characteristic and note that the conjecture in arbitrary dimension would follow from strict $\bb A^1$-invariance.
\end{abstract}

\tableofcontents

\section{Introduction}
In \cite{HuebnerSchmidt2021}, Hübner-Schmidt define the tame site of a scheme. The tame site is finer than the Nisnevich site but coarser than the \'etale site and if $X$ is a scheme over a fixed base scheme $S$, then there are natural morphisms of sites
$$X_{\et}\xrightarrow{\alpha}(X/S)_t \xrightarrow{\beta} X_{\Nis} \quad \quad X_{\et}\xrightarrow{\epsilon}X_{\Nis},$$ the tame site of $X$ over $S$ being denoted by $(X/S)_t$. In this article we study how various $p$-adic cohomology theories change when restricted along $\alpha$ and $\beta$. Hübner-Schmidt show that the \'etale cohomology groups of $X$ coincide with the tame cohomology groups of $X$ if $X$ is proper over $S$ or if the coefficients are invertible on $S$. Discarding wild phenomena, the tame site has some nice properties which do not hold for the \'etale site: assuming resolution of singularities, the tame cohomology groups are homotopy invariant for all locally constant torsion sheaves in positive characteristic, and one expects full and not just semi-purity for logarithmic deRham-Witt sheaves, the case of dimension one being shown in \cite{Huebner2023}.
Logarithmic deRham-Witt sheaves give the right $p$-adic \'etale motivic cohomology theory in characteristic $p>1$, and naturally are sheaves in the tame topology. The $p$-adic \'etale Tate twists give the right $p$-adic \'etale motivic cohomology in mixed characteristic $(0,p)$.
One of the goals of this article is to define and study their tame analogue by pushing them forward to the tame site. 
In the Zariski and Nisnevich topology, motivic cohomology can be defined using cycle complexes and in the following, for a scheme $X$ of finite type over a base scheme $S$ which is the spectrum of a field or of a discrete valuation ring, we denote by $\bb Z(n)_X$ the motivic complexes defined by Bloch and Levine \cite{Bl86,Le01}. Since $p$-adic \'etale Tate twists are constructed by glueing vanishing cycles with logarithmic deRham-Witt sheaves, we begin with a detailed study of the  characteristic $p$ situation. Our first result reads as follows:
\begin{theorem}
Let $k$ be a perfect field of characteristic $p>0$ and $X$ be a smooth $k$-scheme. We denote by $W_r\Omega_{X,\log,t}^n$ the logarithmic deRham-Witt sheaves in the tame topology.
\begin{enumerate}
\item (Beilinson-Lichtenbaum; Prop. \ref{proposition_BL_dRW}) There is an isomorphism
$$\bb  Z(n)_X/p^r \stackrel{\cong}{\longrightarrow}\tau_{\leq n}R\beta_*(W_r\Omega_{X,\log,t}^n[-n])$$
in $D^b(X_{\Nis},\bb Z/p^r\bb Z).$
\item ((Semi-)Purity; Prop. \ref{proposition_semi_purity_logdeRhamWiss} and \ref{proposition_purity_curves}) Let $i : Z \hookrightarrow X$ be a locally closed regular subscheme
          of characteristic $p$ and of codimension
          $c \,(\ge 1)$.
         Then there is a Gysin isomorphism
$$
     W_r\Omega_{Z,\log,t}^{n-c} \stackrel{\cong}{\longrightarrow} \tau_{\leq 0} R i^!W_r\Omega_{X,\log,t}^{n}     
$$
in $D^b((Z/\Spec(k))_{t},\bb Z/p^r\bb Z)$. If $\dim X=1$, then full purity holds, i.e. there is an isomorphism $
     W_r\Omega_{Z,\log,t}^{n-c} \xto{\cong} R i^!W_r\Omega_{X,\log,t}^{n}$ in $D^b((Z/\Spec(k))_{t},\bb Z/p^r\bb Z).$
\end{enumerate}
\end{theorem}
\noindent As mentioned already, one expects to be able to do without the truncation in $(ii)$ even in higher dimension; this is work in progress by Koubaa.
We now turn to the mixed characteristic situation and fix the following notation.
Let $A$ be a discrete valuation ring with perfect residue field $k$ of characteristic $p>0$, local parameter $\pi$ and fraction field $K$ of characteristic zero. Let $X$ be a regular semistable scheme which is flat of finite type over $S=\Spec A$. Let $i:Y\hookrightarrow X$ be the inclusion of the special fiber of $X$ and $j:X_K\hookrightarrow X$ the inclusion of the generic fiber. We also denote by $j$ the morphism of sites $(X_K/K)_t \xrightarrow{} (X/S)_\t$.
\begin{definition}
We define \textit{$p$-adic tame Tate twists}
$\T_r^t(n)_X:=\tau_{\leq n}R\alpha_*\T_r(n)_X\in D^b((X/S)_{t},\bb Z/p^r\bb Z).$
\end{definition}
Their main properties, which we prove in this article, are summarised in the following theorem: 
\begin{theorem}
\begin{enumerate}
\item There is a distinguished triangle of the form
$$i_*\nu^{n-1}_{Y,r,t}[-n-1]\xto{g} \T^t_r(n)_X\xto{t'}\tau_{\leq n}Rj_*\mu_{p^r}^{\otimes n}\xto{\sigma_{X,r}(n)}i_*\nu^{n-1}_{Y,r,t}[-n]\in D^b((X/S)_{t},\bb Z/p^r\bb Z)$$
in which $\T_r^t(n)_X$ is concentrated in $[0,n]$ and the triple $(g,\T_r^t(n)_X,t')$ is unique up to unique isomorphism and $g$ is determined by the pair $(\T_r^t(n)_X,t')$. For the definition of $\nu^{n-1}_{Y,r,t}$ see Definition \ref{definition_nu_X}. If $X$ is smooth over $S$, then $\nu^{n-1}_{Y,r,t}$ coincides with  $W_r\Omega_{X,\log,t}^{n-1}$. The map $\sigma_{X,r}(n)$ is Kato's residue map (see Section \ref{section_etaletatetwists}).
\item (Beilinson-Lichtenbaum; Prop. \ref{proposition_BL}) Let $X$ be smooth over $S$. Then the change of sites $(X/S)_t \xrightarrow{\beta} X_{\Nis}$ induces an isomorphism
$$\bb  Z(n)_X/p^r\xto{\cong}\tau_{\leq n}R\beta_*\T_r^t(n)_X. $$
\item (Semi-purity; Prop. \ref{proposition_semi_purity_Tate_twists}) Let $i : Z \hookrightarrow X$ be a locally closed regular subscheme
          of characteristic $p$ and of codimension
          $c \,(\ge 1)$.
         Then there is a Gysin isomorphism
$$
\begin{CD}
     W_r\Omega_{Z,\log,t}^{n-c}[-n-c] @>{\simeq}>> \tau_{\leq n+c} R i^!\T^t_r(n)_X
      \quad \hbox{ in } \; D^b((Z/\Spec(k))_{t},\bb Z/p^r\bb Z).
\end{CD}
$$
\item (Product structure; Prop. \ref{proposition_product_structure}) There is a morphism
$$\T_r^t(n)_X\otimes^{\bb L}\T_r^t(m)_X\to \T_r^t(n+m)_X  $$
in $D^b((X/S)_{t},\bb Z/p^r\bb Z)$ that extends the natural isomorphism $\mu_{p^r}^{\otimes n}\otimes \mu_{p^r}^{\otimes m}\cong \mu_{p^r}^{\otimes n+m}$ on $X[\frac{1}{p}]$.
\end{enumerate}
\end{theorem}

The Beilinson-Lichtenbaum conjecture for logarithmic deRahm-Witt sheaves and $p$-adic Tate twists both in the \'etale and the tame topology tells us that their main difference lies in cohomological degrees above the weight. Due to the expected full purity in the tame topology, one can expect a very nice form of the Gersten conjecture in the tame topology (see Conjecture \ref{conjecture_Gersten_logdRW} and \ref{conjecture_Gersten_tame_Tate_twists}) which we verify for curves in characteristic $p$.
\begin{proposition}(Prop. \ref{proposition_Gersten_curve})
Let $X$ be a smooth curve over a field $k$ of characteristic $p>0$. Let $x\in X$ be a closed point and $A:=\roi_{X,x}$. Then there are short exact sequences
$$ 0  \to W_r\Omega_{X,\log,t}^n(A) \to W_r\Omega_{X,\log,t}^n(K(A))\to W_r\Omega_{k(x),\log,t}^{n-1}(k(x))\to 0$$
and
$$0\to H^1_t(\Spec A/k,W_r\Omega_{A,\log,t}^n)\to  H^1_t(\Spec K(A)/k,W_r\Omega_{X,\log,t}^n) \to  H^{1}_t(\Spec k(x)/k,W_r\Omega_{k(x),\log,t}^{n-1}) \to 0 $$
and isomorphisms $H^j_t(\Spec A/k,W_r\Omega_{X,\log,t}^n)\xto{\cong}  H^j_t(\Spec K(A)/k,W_r\Omega_{X,\log,t}^n)$
for $j\geq 2$.
\end{proposition}

The structure of the article is as follows. In Section \ref{section_tame_site} we recall the tame site and the most important theorems concerning it. We also prove a new rigidity result for tame morphisms. In Section \ref{section_homological_algebra} we collect the main technical tools from homological algebra which we will need. In Section \ref{section_logdeRhamWitt} we prove our main results about logarithmic deRham-Witt sheaves. The central result is the Gersten conjecture for curves. We expect that these results can be improved using work in progress by Koubaa.
In Section \ref{section_etaletatetwists} we recall Saito's $p$-adic \'etale Tate twists.
In Section \ref{section_tametatetwists} we define $p$-adic tame Tate twists and prove that they share many important properties with $p$-adic \'etale Tate twists.
In Section \ref{section_cycle_class} we discuss the mod $p^r$ cycle class map from Chow groups to tame cohomology.

\paragraph{Acknowledgement.} First of all, I would like to heartily thank Rızacan \c{C}ilo\u{g}lu for many discussions and help with this work. I would also like to thank Christian Dahlhausen, Amine Koubaa and Katharina Hübner for their comments, interest and helpful discussions. The author  acknowledges support by Deutsche Forschungsgemeinschaft (DFG, German Research Foundation) through the Collaborative Research Centre TRR 326 \textit{Geometry and Arithmetic of Uniformized Structures}, project number 444845124.

\section{The tame site}\label{section_tame_site}
\subsection{Definitions, facts and conjectures}
We recall the definition of the tame site and some facts on tame cohomology from \cite{HuebnerSchmidt2021}. We fix the following conventions. A valuation $v$ on a field $K$ is always non-archimedean, possibly trivial and not necessarily discrete or of finite rank. We denote by $\roi_v, \mathfrak{m}_v$ and $k(v)$ the corresponding valuation ring, maximal ideal and residue field.

\begin{definition}
An $S$-valuation on an $S$-scheme $X$ is a valuation $v$ on the residue field $k(x)$ of some point $x\in X$ such that there exists a morphism $\phi:\Spec(\roi_v)\to S$ making the diagram
$$\xymatrix{
\Spec(k(x)) \ar[r] \ar[d]_{} & X \ar[d]^{f} \\
\Spec(\roi_v) \ar[r]^-\phi  & S
}$$
commute. The set of all $S$-valuations is denoted by $\Val_SX$. We denote  elements of $\Val_SX$ by $(x,v),x\in X,v\in \Val_S(k(x))$.
\end{definition}
Given a commutative square
$$\xymatrix{
X' \ar[r]^-\phi \ar[d]_{f'} & X \ar[d]^{f} \\
S'\ar[r]^\psi  & S
}$$
there is an associated map $$\Val_\psi\phi: \Val_{S'}X'\to \Val_SX, \quad (x',v')\mapsto (x=\phi(x'),v=v'\mid_{k(x)}).$$
\begin{definition}
The tame site $(X/S)_t$ consists of the following data:
\begin{enumerate}
\item The underlying category $\mathrm{Cat}(X/S)_t$ is the category of \'etale morphisms $p:U\to X$.
\item A family $(U_i\to U)_{i\in I}$ of morphisms in $\mathrm{Cat}(X/S)_t$ is a covering if it is an \'etale covering and for every point $(u,v)\in \Val_SU$ there exists an index $i$ and a point $(u_i,v_i)\in \Val_SU_i$ mapping to $(u,v)$ such that $v_i/v$ is tamely ramified.
\end{enumerate}
When $S$ is the spectrum of a field for simplicity we sometimes denote $(X/S)_t$ by $X_t$.
\end{definition}

\begin{remark}
The Nisnevich topology has the same underlying category as the tame topology. A family $(U_i\to U)_{i\in I}$ of morphisms in $\mathrm{Cat}X_{\Nis}$ is a Nisnevich covering if for every $x\in U$ there exists an $i$ and a point $y\in U_i$ such that $k(x)\cong k(y)$. In particular for $(x,v)\in \Val_SU$ there exists a tamely ramified preimage $(y,v)\in \Val_SU_i$.
In other words,	every Nisnevich covering is tame and there are natural morphisms of sites
$X_{\et}\xrightarrow{\alpha}(X/S)_\t \xrightarrow{\beta} X_{\Nis} $ and $X_{\et}\xrightarrow{\epsilon}X_{\Nis}.$
\end{remark}

There are the following two important comparison theorems:
\begin{theorem}\cite[Prop.
	8.1]{HuebnerSchmidt2021}\label{theorem_tame_coh_l_coef}
	Let $m\geq 1$ be an integer invertible on $S$ and $F$ be an \'etale sheaf of $\bb Z/m\bb Z$-modules on $X$. Then the natural morphism of sites $\alpha: X_{\et}\to (X/S)_t$ induces isomorphisms
$$H^q_t(X/S,\alpha_*F)\xto{\cong} H^q_{\et}(X, F)$$
for all $q\geq 0$.
\end{theorem}

\begin{remark}
By the Leray spectral sequence it suffices to show that $R^q\alpha_*F=0$ for $q\geq 1$; this is shown in \textit{loc. cit}. The theorem may therefore be rephrased saying that the canonical map $\alpha_*F\to R\alpha_*F$ is a quasi-isomorphism.
\end{remark}

\begin{theorem}\cite[Thm. 8.2]{HuebnerSchmidt2021}\label{theorem_tame_coh_proper}
	Let $X$ be a quasi-compact scheme having the property that every finite subset of $X$ is contained in an affine open, $S$ a quasi-compact quasi-separated scheme and $X\to S$ a proper morphism. Then for every sheaf $F$ of abelian groups on $(X/S)_t$ with \'etale sheafification $F^{\sharp\et}$ the natural map
$$H^q_t(X/S,F)\to H^q_{\et}(X, F^{\sharp\et})$$
is an isomorphism for all $q\geq 0$.
\end{theorem}

Tame cohomology groups are expected to be $\bb A^1$-invariant in the following sense.
\begin{conjecture}
Let $X$ be a scheme over a base scheme $S$ and $pr:\bb A^1_X\to X$ the natural projection. Then for every torsion sheaf $F\in Sh_t(X/S)$ the natural map
$$H^q_t(X/S,F)\to H^q_t(\bb A^1_X/S, pr^*F)$$
is an isomorphism for all $q\geq 0$.
\end{conjecture}

\begin{theorem}\cite[Thm. 15.4]{HuebnerSchmidt2021}\label{theorem_A1_invariance_tame_coh}
Let $S$ be an affine noetherian scheme of characteristic $p>0$	and $X$ a regular scheme which is essentially of finite type over $S$ and $pr:\bb A^1_X\to X$ the natural projection. Assume that resolution of singularities holds over $S$. Then for every locally constant torsion sheaf $F\in Sh_t(X/S)$ the natural map
$$H^q_t(X/S,F)\to H^q_t(\bb A^1_X/S, pr^*F)$$
is an isomorphism for all $q\geq 0$.
\end{theorem}

Finally, one expects full purity for locally constant coefficients and logarithmic deRham-Witt sheaves (see \cite[p.2]{Huebner2021}).
\begin{conjecture}\label{conjecture_purity}
Let $X$ be a regular scheme of finite type over $S$ with $S$ pure of characteristic $p>0$. Then the following statements hold.
\begin{enumerate}
\item For every locally constant $p$-torsion sheaf $F\in Sh_t(X/S)$ and $U\subset X$ a dense open subset, the natural map
$$H^q_t(X/S,F)\to H^q_t(U/S,F\mid_U)$$
is an isomorphism for all $q\geq 0$.
\item For $Z\subset X$ a regular closed subscheme of codimension $c$, there is an isomorphism
$$W_r\Omega_{Z,\log,t}^{n-c}\xto{\cong} Ri^! W_r\Omega_{X,\log,t}^{n}[c]$$
in $D^b((Z/S)_{t},\bb Z/p^r\bb Z)$.
\end{enumerate}
\end{conjecture}

\begin{theorem}\cite[Thm. 15.2]{HuebnerSchmidt2021}
Let $S$ be an affine noetherian scheme of characteristic $p>0$ and $X$ be a regular scheme which is separated and essentially of finite type over $S$. Assume resolution of singularities holds over $S$. Then for any open dense subscheme $U\subset X$ and every locally constant $p$-torsion sheaf $F\in Sh_t(X/S)$ the natural map
$$H^q_t(X/S,F)\to H^q_t(U/S,F\mid_U)$$
is an isomorphism for all $q\geq 0$.
\end{theorem}
We will say more about Conjecture \ref{conjecture_purity}(ii) in Section \ref{section_logdeRhamWitt}.

\subsection{Rigidity}
We recall a few facts about henselian pairs.
\begin{definition}
\label{def:henspair}
A \emph{pair} is the data $(S, I)$ where $S$ is a commutative ring and $I \subset S$ is
an ideal. A pair $(S, I)$ is said to be \emph{henselian} if the following equivalent
(cf.~\cite[Tag 09XD]{stacks-project} for the equivalence) conditions hold:
\begin{enumerate}
\item
Given a polynomial $f(x) \in S[x]$ and a root $\overline{\alpha} \in S/I$ of $\overline{f} \in (S/I)[x]$ with
$\overline{f}'(\alpha)$ being a unit of $S/I$, then $\overline \alpha$ lifts to a root $\alpha \in S$ of $f$.
\item The ideal $I$ is contained in the Jacobson radical of $S$, and the same
condition as (1) holds for \emph{monic} polynomials  $f(x) \in S[x]$.
\item Given any commutative diagram
\begin{equation} \label{lifthens}  \xymatrix{
A \ar[d] \ar[r] &  S \ar[d]  \\
B \ar[r] \ar@{-->}[ru] &  S/I\
}\end{equation}
with $A \to B$ \'etale, there exists a lift as in the dotted arrow.

\end{enumerate}

If $S$ is local with maximal ideal $\frak m$, then $S$ is said
to be a \emph{henselian local ring} if $(S, \frak m)$ is a henselian pair.

We call a map of rings $f:R_2\to R_1$ a henselian surjection if $f$ is surjective and $(R_2,\mathrm{ker}(f))$ is a henselian pair.
\end{definition}

\begin{remark}\label{remark_ideals_hs}
\begin{enumerate}
\item If $I$ and $J$ are ideals of $S$ and $I\subset J$, then $(S,I)$ is a henselian pair if $(S,J)$ is a henselian pair.
\item If $(R,I)$ is a henselian pair, and $R\to S$ an integral ring homomorphism, then the pair $(S,IS)$ is henselian.
\end{enumerate}
\end{remark}


From now on let $S$ be an integral, pure dimensional separated and excellent scheme and let $X$ be a noetherian scheme separated and essentially of finite type over $S$.
Note that we do not assume $X\to S$ to be of finite type. For an integral $X$ we let
$$\dim_SX:= \mathrm{deg.tr.}(k(X)/k(T))+\dim_{\rm  Krull}T,$$
where $T$ is the closure of the image of $X$ in $S$. We call $X$ an $S$-curve if $\dim_SX=1$ and for a regular connected $S$-curve $C$ we let $\bar{C}$ be the unique regular compactification of $C$ over $S$. For the following definition see \cite{KerzSchmidt2010}.
\begin{definition}
Let $Y\to X$ be an \'etale cover of separated, essentially of finite type $S$-schemes. We say that $Y\to X$ is curve-tame if for any morphism $C\to X$ with $C$ a regular $S$-curve, the base change $Y\times_XC\to C$ is tamely ramified along $\bar{C}\setminus C$. 
\end{definition}

\begin{proposition}\label{proposition_rigidity}
Let $S$ be a henselian discrete valuation ring with maximal ideal $\mathfrak{m}_S=(\pi)$. Let $A$ be a separated essentially of finite type $S$-algebra such that $(A,I=(\pi)A)$ is a henselian pair. Then the functor
$$B\mapsto B\otimes_AA/I$$
induces an equivalence between the category of finite curve-tame covers $\Spec(B)\to \Spec(A)$ and finite curve-tame covers $\Spec(B/I)\to \Spec(A/I)$.
\end{proposition}
\begin{proof}
By \cite{GabberAffineAnalogue} there is an equivalence between the category of finite \'etale covers $\Spec(B)\to \Spec(A)$ and finite \'etale covers $\Spec(B/I)\to \Spec(A/I)$.
We have to show that curve-tameness is preserved under this equivalence. But since $A$ is henselian along $I$, by Definition \ref{def:henspair}(iii) we have that for any regular $S$-curve $C$ and morphism $\varphi:C\to \Spec(A)$ we have that either $im(\varphi)$ is contained in $\Spec(A/I)$ or $\overline{im(\varphi)}$ is contained in $\Spec(A)$. Therefore curve-tameness can be checked by $S$-curves $C$ and morphisms $\varphi:C\to \Spec(A)$ whose image is contained in $\Spec(A/I)$.
\end{proof}

\begin{corollary}
Let $S$ be a henselian discrete valuation ring with maximal ideal $\mathfrak{m}_S=(\pi)$. Let $A$ be a separated essentially of finite type $S$-algebra such that $(A,I=(\pi)A)$ is a henselian pair with $A$ and $A/I$ geometrically unibranch. Then the functor
$$B\mapsto B\otimes_AA/I$$
induces an equivalence between the category of finite tame Galois covers $\Spec(B)\to \Spec(A)$ and finite tame Galois covers $\Spec(B/I)\to \Spec(A/I)$.
\end{corollary}
\begin{proof}
It is shown in the proof of \cite[Prop. 5.2]{HuebnerSchmidt2021} that a finite \'etale Galois cover is tame if and only if it is curve tame. The corollary therefore follows from Proposition \ref{proposition_rigidity}. 
\end{proof}

\begin{conjecture}\label{conjecture_affine_anal_pbs}
Let $S$ be a henselian discrete valuation ring with maximal ideal $\mathfrak{m}_S=(\pi)$. Let $A$ be a separated essentially of finite type $S$-algebra such that $(A,I=(\pi)A)$ is a henselian pair. Let $i:\Spec A/I\to \Spec A$ be the corresponding inclusion. Let $F\in Sh((\Spec A/S)_t)$. Then the restriction map
$$H^{s}(\Spec A,F)\cong H^{s}(\Spec A/I,i^*F)$$
is an isomorphism for all $s$.
\end{conjecture}


\section{Some homological algebra}\label{section_homological_algebra}
Let $A$ be a discrete valuation ring with residue field $k$ and fraction field $K$.
Let $Y$ be the special fiber of $X$ and $X_K$ the generic fiber.
In this section we consider the commutative diagram
$$\xymatrix{
\bb A^1_{X_K} \ar@{^{(}->}[r]^{j'} \ar[d]_{\pi_K} & \bb A^1_{X}\ar[d]^{\pi} & \bb A^1_{Y} \ar@{_{(}->}[l]_{i'} \ar[d]^{\pi_k} \\
X_K \ar@{^{(}->}[r]^j  & X  & Y  \ar@{_{(}->}[l]_{i}
}$$
and recollect some facts about the commutativity of the corresponding derived functors and the change of sites (see also \cite[Sec. 2]{Geisser2004}).
We denote by $X_{\et}, (X/S)_t$, $X_{\Nis}$ and $X_{\rm Zar}$ the \'etale, tame, Nisnevich and Zariski sites respectively. In everything that follows, the Nisnevich topology may be replaced by the Zariski topology.
There are natural morphisms of sites
$$X_{\et}\xrightarrow{\alpha}(X/S)_\t \xrightarrow{\beta} X_{\Nis}, \quad \quad X_{\et}\xrightarrow{\epsilon}X_{\Nis},$$
as well as
$$X_{K,\et}\xrightarrow{\alpha}(X_K/K)_t \xrightarrow{j_1}(X_K/S)_t \xrightarrow{j_2} (X/S)_\t, \quad \quad (X_K/K)_t \xrightarrow{j} (X/S)_\t$$
and
$$Y_{\et}\xrightarrow{\alpha}(Y/k)_t \xrightarrow{\cong}(Y/S)_\t \xrightarrow{i} (X/S)_\t.$$
In the following all sheaves will be sheaves of abelian groups and the derived functors will be between the derived categories of sheaves of abelian groups. All statements also hold for $j$ replaced by $j_1$ or $j_2$ and we drop the index for ease of exposition.
When there is no risk of confusion, we write $Rj_*$ (resp. $Ri_*$) for the derived pushforward without mentioning the topology we are working with.

The following proposition is well known for the \'etale and Nisnevich topologies.
\begin{proposition}\label{proposition_exactness_functors}
\begin{enumerate}
	\item The sheafification functors
		$\epsilon^*,\alpha^*,\beta^*$ are exact and
		preserve colimits.
\item The forgetful functors $\epsilon_*,\alpha_*,\beta_*$ are left exact, preserve injectives and limits.
\item There are mutually adjoint functors $i^*\vdash i_*\vdash i^!$ and $j_!\vdash j^*\vdash j_*$. In particular, $i^!,j_*$ are left exact, $i_*,j^*$ are exact and in our situation $i^*,j_!$ are exact. Consequently, $i_*,i^!,j^*,j_*$ preserve limits and injectives, and $i^*,i_*,j_!,j^*$ preserve colimits.
\end{enumerate}
\end{proposition}
\begin{proof}
(i) Since the underlying categories of $X_{\et}$ and $(X/S)_t$ are the same, $\Cat((X/S)_t)$ has finite limits and the $\beta^\sharp: \Cat(X_\Nis)\to \Cat((X/S)_t)$ and $\alpha^\sharp:\Cat((X/S)_t)\to \Cat(X_\et)$ preserve finite limits. The statement follows from \cite[Ch. II, Prop. 2.6, §3]{Mi80}.
	
	(ii) This follows from the functors $\epsilon_*,\alpha_*,\beta_*$ being right adjoint to the functors
		$\epsilon^*,\alpha^*,\beta^*$ and from the fact that these, by (i), are exact.

(iii) Again, this is well-known for the \'etale and Nisnevich topologies. In the tame topology, the arguments are the same. The adjunction $j_!\vdash j^*\vdash j_*$ is shown in \cite[Lem. 2.5]{HuebnerSchmidt2021}. For the adjunction $i^*\vdash i_*\vdash i^!$ one notes that the proof of the triple theorem in the \'etale topology \cite[Ch. II, Thm. 3.10]{Mi80} goes through exactly the same way for the tame topology using that the exactness of tame sheaves can be checked stalkwise \cite[Lem. 2.10]{HuebnerSchmidt2021}. 
\end{proof}

\begin{proposition}\label{proposition_commutation_pushforward}
There are isomorphisms of derived functors
\begin{enumerate}
\item
$\epsilon^*j^*\xto{\sim} j^*\epsilon^* \quad \quad Rj_*R\epsilon_*\xto{\sim} R\epsilon_*Rj_*$
\item $\epsilon^*i^*\xto{\sim} i^*\epsilon^* \quad \quad i_*R\epsilon_*\xto{\sim}R\epsilon_*i_*$
\item $\epsilon^*i_*\xto{\sim} i_*\epsilon^* \quad \quad Ri^!R\epsilon_*\xto{\sim}R\epsilon_*Ri^!$
\item $\epsilon^*j_!\xto{\sim} j_!\epsilon^* \quad \quad j^*R\epsilon_*\xto{\sim}R\epsilon_*j^*$\end{enumerate}
and the same isomorphisms hold for $\epsilon$ replaced by $\alpha$ or $\beta$. Furthermore, there is an isomorphism
\begin{enumerate}
\item[(v)] $i_*'\pi^*_k\cong \pi^*i_*$.
\end{enumerate}
\end{proposition}
\begin{proof}
	(i)-(iv): The pushforward functors appearing on the right preserve injectives by
	Proposition \ref{proposition_exactness_functors},
	therefore by Lemma \ref{lemma_grothendieck_sp_seq}
	(ii), the composition of their right derived
	functors is
	the right derived functor of their composition. This
	implies the functors on the right are adjoint to
	those on left, which exist by the universal property of sheafification. By the uniqueness of adjoints, it
	suffices to check the assertions on left, which can be
	verified quickly from the definitions or on stalks.
	
Lets make the fist step concrete for (iii). By the above argument we have that $Ri^!R\epsilon_*\cong R(i^!\epsilon_*)$ and $R\epsilon_*Ri^!\cong R(\epsilon_*i^!)$. There are two adjunctions
$$\Hom(i_*\epsilon^*(\cal F),\cal G)\cong \Hom(\cal F,R(\epsilon_*i^!)(\cal G)) $$
and
$$\Hom(\epsilon^* i_*(\cal F), \cal G)\cong \Hom(\cal F,R(i^!\epsilon_*)(\cal G)) $$
Therefore, by the uniqueness of adjoints, it suffices to show that $\epsilon^* i_*(\cal F)\cong i_*\epsilon^*(\cal F)$. 

We verify that the morphisms on the left are isomorphisms: (i, ii) The underlying categories \'etale, tame and Nisnevich sites of $X$ (resp. $X_K,Y$) are the same. Therefore 
$$\epsilon^*j^*\cal F\cong\sharp\circ\epsilon^*_p\circ \sharp\circ j^*_p(\cal F)\cong\sharp\circ\epsilon^*_p\circ  j^*_p(\cal F)\cong\sharp\circ  j^*_p\circ\epsilon^*_p(\cal F)\cong\sharp\circ  j^*_p\circ \sharp\circ \epsilon^*_p(\cal F)\cong j^*\epsilon^*\cal F$$
and exactly the same holds for $j^*$ replaced by $i^*$.
(iii,iv) may be checked on stalks: since for $y\in Y_{\et}$ we have that $Y_{\bar{y}}\cong Y\times_SX_{\bar{y}}$, it follows that $(\epsilon^* i_*(\cal F))_x\cong (i_*\epsilon^*(\cal F))_{\bar{x}}\cong \cal F_{\bar{x}}$ for $x\in Y$ and zero otherwise; $(\epsilon^* j_!(\cal F))_x\cong (j_!\epsilon^*(\cal F))_{\bar{x}}\cong \cal F_{\bar{x}}$ for $x\in U$ and zero otherwise. 

(v): Let $\cal F\in Sh(Y)$ and $U\subset \bb A^1_X$. Then, since $\pi$ and $\pi_k$ are open and since $\pi(U)\cap Y=\pi_k(U\cap \bb A_Y^1)$, we have that $i_*'\pi^*_k\cal F(U)\cong \pi^*i_*\cal F(U)$.
\end{proof}

\begin{lemma}\label{lemma_grothendieck_sp_seq}
Let $\cal A,\cal B,\cal C$ be abelian categories. Let $F:\cal A\to \cal B$ and $G:\cal B\to\cal C$ be left exact functors. Assume that $\cal A,\cal B$ have enough injectives. Then the following are equivalent:
\begin{enumerate}
\item $F(I)$ is right acyclic for $G$ for each injective object $I$ of $\cal A$, and
\item the canonical map $t:R(G\circ F)\to RG\circ RF$ is an isomorphism of functors from $D^+(\cal A)$ to $D^+(\cal C)$.
\end{enumerate}
Furthermore, $(i)$ implies that
\begin{enumerate}
\item[(iii)] the canonical map $t:\tau_{\leq n}R(G\circ F)\to \tau_{\leq n}RG\circ RF\leftarrow \tau_{\leq n}RG\circ \tau_{\leq n}RF$ are isomorphism of functors from $D^+(\cal A)$ to $D^+(\cal C)$.
\end{enumerate}
\end{lemma}

\begin{proof}
We just show that $(i)$ implies $(iii)$. The first equivalence is well known. Let $A\in D^+(\cal A)$ and $A\to I^\bullet$ be an injective resolution. Then there is a sequence of isomorphisms $$\tau_{\leq n}R(G\circ F)(A)\stackrel{(i)}{\cong} \tau_{\leq n}RG(RF(A))\cong \tau_{\leq n}RG(F(I^\bullet))\xleftarrow{\cong} \tau_{\leq n}RG(\tau_{\leq n}F(I^\bullet))\cong \tau_{\leq n}RG(\tau_{\leq n}RF(A))$$
in which the middle arrow is an isomorphism since $F(I^i)$ is acyclic for $G$ for all $i$. 
\end{proof}

Let $\cal A$ be an abelian category with enough injectives and $D(\cal A)$ be the derived category of complexes of objects of $\cal A$.
\begin{lemma}\cite[Lem. 2.1.2]{Sa07} \label{lemma_hom_D}
Let $m$ and $q$ be integers. Let $\cal K$ be an object of $D(\cal A)$ concentrated in degrees $\leq m$ and $\cal K'$ be an object of $D(\cal A)$ concentrated in degrees $\geq 0$. Then we have that
 $\Hom_{D(\cal A)}(\cal K,\cal K'[-q])=\Hom_{\cal A}(\cal H^m(\cal K),\cal H^0(\cal K'))$ if $q=m$ and $\Hom_{D(\cal A)}(\cal K,\cal K'[-q])=0$ if $q>m$.
\end{lemma}

The following lemma allows to make uniqueness assertions in the derived category:
\begin{lemma}\cite[Lem. 2.1.2]{Sa07} \label{lemma_uniqueness_D}
Let $\cal N_1\xto{f} \cal N_2\xto{g} \cal N_3\xto{h} \cal N_1[-1]$ be a distinguished triangle in $D(\cal A)$.
\begin{enumerate}
\item Let $i:\cal K\to \cal N_2$ be a morphism with $g\circ i=0$ and suppose that $\Hom_{D(\cal A)}(\cal K,\cal N_3[-1])=0$. Then there exists a unique morphism $i':\cal K\to \cal N_1$ that $i$ factors through.
\item Let $p:\cal N_2\to \cal K$ be a morphism with $p\circ f=0$ and suppose that $\Hom_{D(\cal A)}(\cal N_1[1],\cal K)=0$. Then there exists a unique morphism $p':\cal N_3\to \cal K$ that $p$ factors through.
\item Suppose that $\Hom_{D(\cal A)}(\cal N_2,\cal N_1)=0$ (resp. $\Hom_{D(\cal A)}(\cal N_3,\cal N_2)=0$). Then relatively to a fixed triple $(\cal N_1,\cal N_3,h)$, the other triple $(\cal N_2,f,g)$ is unique up to unique isomorphism, and $f$ (resp. $g$) is determined by the pair $(\cal N_2,g)$ (resp. $(\cal N_2,f)$).
\end{enumerate}
\end{lemma}

\section{Results for logarithmic deRham-Witt sheaves}\label{section_logdeRhamWitt}
\subsection{Logarithmic deRham-Witt sheaves in various topologies}
We closely follow the exposition in \cite[Sec. 1.2]{Morrow2019}.
Let $X$ be an $\bb F_p$-scheme, and let $\tau$ denote the Zariski, Nisnevich, tame, or \'etale topology.
Considering the Hodge-Witt sheaf $W_r\Omega^n_X$ as a sheaf in the $\tau$-topology, we denote
$$W_r\Omega^n_{X,\log,\tau}\subset W_r\Omega^n_X$$
the subsheaf which is generated $\tau$-locally by $\dlog$ forms, that is, the $\tau$-sheafification of the image of the map of presheaves
$$\dlog[\cdot]: \bb G^{\otimes n}_{m,X}\to W_r\Omega^n_X,\quad a_1\otimes \dots \otimes a_n\mapsto \dlog[a_1]\dots \dlog[a_n],$$
where $[f]\in W_r(A)$, for a ring $A$, is the Teichm\"uller lift of any $f\in A^\times$ \cite[p. 505 (1.1.7)]{Il79} and $\dlog{[f]}:=\tfrac{d[f]}{[f]}$.


\begin{proposition}\label{proposition_logdRW_in_various_top}
If $X$ is a regular $\bb F_p$-scheme, then the inclusion of tame sheaves
$W_r\Omega^n_{X,\log,t}\to \alpha_*W_r\Omega^n_{X,\log,\et} $ is an equality.
\end{proposition}
\begin{proof}
Let $A$ be the stalk at a point of $X$ in the tame topology. Then $A$ is in particular a regular henselian local ring and we have a commutative diagram
$$\xymatrix{
W_r\Omega^n_{A,\log,t} \ar@{^{(}->}[r]^{}  & \alpha_*W_r\Omega^n_{A,\log,\et}  \\
 {K}^M(A)/p^r \ar@{->>}[r] \ar[u]  &  \hat{K}^M(A)/p^r \ar[u]_\cong
}$$
in which the isomorphism on the right follows from the Gersten conjecture for log deRham-Witt sheaves \cite{GrosSuwa1988} and Milnor K-theory \cite{Kerz2009,Kerz2010} and the Bloch-Kato-Gabber theorem \cite{BlochKato1986}. The lower horizontal map is surjective by \cite{Kerz2010}.
\end{proof}

\begin{remark}
Proposition \ref{proposition_logdRW_in_various_top} is shown with the tame topology replaced by the Nisnevich topology in \cite{Kato1982} and with the tame topology replaced by the Zariski topology in \cite[Thm. 1.2]{Morrow2019}.
\end{remark}

\subsection{The Beilinson-Lichtenbaum conjecture}\label{subsection_BL_dRW}
Let $k$ be a perfect field of characteristic $p>0$ and $X$ be a smooth $k$-scheme. The Beilinson-Lichtenbaum conjecture in characteristic $p$ and with $p$-coefficients says that there is an isomorphism 
$$\bb  Z(n)_X/p^r \stackrel{\cong}{\longrightarrow}\tau_{\leq n}R\epsilon_*W_r\Omega_{X,\log,\et}^n[-n]$$
and is shown in \cite[Thm. 8.3]{GeisserLevine2000}. The analogous statement comparing motivic cohomology in the tame and Zariski/Nisnevich topology is the following:
\begin{proposition}\label{proposition_BL_dRW} Let the notation be as above.
Then there is an isomorphism
$$\bb  Z(n)_X/p^r \stackrel{\cong}{\longrightarrow}\tau_{\leq n}R\beta_*(W_r\Omega_{X,\log,t}^n[-n])$$
in $D^b(X_{\Nis},\bb Z/p^r)$.
\end{proposition}
\begin{proof}
There is a sequence of isomorphisms
$$\tau_{\leq n}R\beta_*W_r\Omega_{X,\log,t}^n[-n]\stackrel{Prop. \ref{proposition_logdRW_in_various_top}}{\cong}\tau_{\leq n}R\beta_*\tau_{\leq n}R\alpha_*W_r\Omega_{X,\log,\et}^n[-n]$$ 
$$\stackrel{Lem. \ref{lemma_grothendieck_sp_seq}(iii)}{\cong} \tau_{\leq n}R\epsilon_*W_r\Omega_{X,\log,\et}^n[-n] \stackrel{\cong}{\longleftarrow} \bb  Z(n)_X/p^r.$$
The last isomorphism is, as mentioned above, shown in \cite[Thm. 8.3]{GeisserLevine2000}.
\end{proof}

\subsection{Semi-Purity}
We first recall known semi-purity results in the \'etale topology for logarithmic deRham-Witt sheaves. Let $k$ be a perfect field of characteristic $p>0$ and $X$ be a smooth $k$-scheme. Let $i : Z \hookrightarrow X$ be a locally closed regular subscheme of codimension $c \,(\ge 1)$. Then is is shown in \cite{Milne1986} that there is a Gysin isomorphism
$$
     W_r\Omega_{Z,\log,\et}^{n-c} \stackrel{\cong}{\longrightarrow} \tau_{\leq 0} R i^!W_r\Omega_{X,\log,\et}^{n}[c]$$
in $D^b(Z_\et,\bb Z/p^r\bb Z),$
that is
$$R^q i^!W_r\Omega_{X,\log,\et}^{n}\cong W_r\Omega_{Z,\log,\et}^{n-c}$$
for $q=c$ and $R^q i^!W_r\Omega_{X,\log,\et}^{n}=0$ for $q< c$. The following proposition shows that the same statement holds in the tame topology.
\begin{proposition}(Semi-Purity)\label{proposition_semi_purity_logdeRhamWiss} Let the notation be as above.
         Then there is a Gysin isomorphism
$$
     W_r\Omega_{Z,\log,t}^{n-c} \stackrel{\cong}{\longrightarrow} \tau_{\leq 0} R i^!W_r\Omega_{X,\log,t}^{n}[c] $$
in $D^b((Z/\Spec(k))_{t},\bb Z/p^r\bb Z).$
\end{proposition}
\begin{proof}
We apply $\tau_{\leq 0}R\alpha_*$ to the isomorphism $W_r\Omega_{Z,\log,\et}^{n-c} \xto{\cong} \tau_{\leq 0} R i^!W_r\Omega_{X,\log,\et}^{n}[c]$ in $D^b(Z_{\et},\bb Z/p^r\bb Z)$ and get an isomorphism $$W_r\Omega_{Z,\log,t}^{n-c}=\tau_{\leq 0}R\alpha_*W_r\Omega_{Z,\log,\et}^{n-c}\xto{\cong}\tau_{\leq 0}R\alpha_*\tau_{\leq 0} R i^!W_r\Omega_{X,\log,\et}^{n}[c]$$ in $D^b((Z/\Spec(k))_{t},\bb Z/p^r\bb Z)$. Finally, by Lemma \ref{lemma_grothendieck_sp_seq}(iii) and Proposition \ref{proposition_commutation_pushforward}(iii) the last term is isomorphic to $\tau_{\leq 0} R i^!W_r\Omega_{X,\log,t}^{n}[c]$.
\end{proof}
From here on all cohomology groups whose coefficients are tame logarithmic deRham-Witt sheaves will be tame cohomology groups with respect to the base scheme $\Spec(k)$.
\begin{corollary}\label{corollary_purity}
Let the notation be as above. \begin{enumerate}
\item There are isomorphisms 
$$ H^{0}(Z,W_r\Omega_{Z,\log,t}^{n-c})\cong H^{c}_Z(X,W_r\Omega_{X,\log,t}^{n}).$$
\item Assuming the purity Conjecture \ref{conjecture_purity}(ii) holds, there are isomorphisms
$$H^i(Z,W_r\Omega^{n-c}_{Z,\log,t})\cong H^{i+c}_Z(X,W_r\Omega_{X,\log,t}^{n}).$$
\item If $x\in X$ is a point of codimension $c$, then, under the assumption of (ii), there are isomorphisms $$H^{i+c}_x(X,W_r\Omega_{X,\log,t}^{n})\cong H^{i}(x,W_r\Omega^{n-c}_{x,\log},t).$$
\end{enumerate} 
\end{corollary}
\begin{proof}

We consider the Grothendieck spectral sequence
$$E_2^{u,v}=H^{u}(Z,R^vi^!W_r\Omega_{X,\log,t}^{n})\Longrightarrow R^{v+u}(\Gamma_Z\circ i^!)W_r\Omega_{X,\log,t}^{n}=:H^{v+u}_Z(X,W_r\Omega_{X,\log,t}^{n}).$$
By Proposition \ref{proposition_semi_purity_logdeRhamWiss}  
this spectral sequence looks as follows:
$$\xymatrix{
c+1  &  H^0(Z,R^{c+1}i^!W_r\Omega_{X,\log,t}^{n}) \ar[drr] & H^1(Z,R^{c+1}i^!W_r\Omega_{X,\log,t}^{n}) & \dots & \\
c  & H^0(Z,W_r\Omega^{n-c}_{Z,\log,t}) \ar[drr] & H^1(Z,W_r\Omega^{n-c}_{Z,\log,t}) & \dots  &\\
c-1  & 0 & 0 & 0 & \dots
}$$
This implies that $H^0(Z,W_r\Omega^{n-c}_{Z,\log,t})\cong H^{c}_Z(X,W_r\Omega_{X,\log,t}^{n})$. 
If the purity Conjecture \ref{conjecture_purity}(ii) holds, then all the terms above the $c$-line are zero and we get isomorphisms $H^i(Z,W_r\Omega^{n-c}_{Z,\log,t})\cong H^{i+c}_Z(X,W_r\Omega_{X,\log,t}^{n})$.

Passing to the colimit over all neighbourhoods $U$ of a point $x$, we get that
$$H^{i+c}_x(X,W_r\Omega_{X,\log,t}^{n}):=\indlim_{U} H^{i+c}_{\overline{\{x\}}\cap U}(U,W_r\Omega_{X,\log,t}^{n})\stackrel{(ii)}{\cong}  \indlim_{V\subset \overline{\{x\}}}H^{i}(V,W_r\Omega^{n-c}_{V,\log,t})=:H^{i}(x,W_r\Omega^{n-c}_{x,\log,t}).$$
\end{proof}

In the next proposition we show that full purity holds for curves. We closely follow the proof of Hübner given in \cite[Sec. 4]{Huebner2023} in the context of adic spaces. Full purity in arbitrary dimension is work in progress by Koubaa.
\begin{proposition}\label{proposition_purity_curves}
Let $X$ be a smooth $k$-scheme of dimension $1$ and $x\in X$ a closed point with closed immersion $i:x\hookrightarrow X$.
Then $$R^ji^! W_r\Omega_{X,\log,t}^n=0$$ for $j\neq 1$.
\end{proposition}
\begin{proof}
We need to show that the stalk
$$(R^ji^! W_r\Omega_{X,\log,t}^n)_{\bar x}= \indlim H^0(x_i,R^ji^! W_r\Omega_{X,\log,t}^n )=H^0(\bar x,R^ji^! W_r\Omega_{X,\log,t}^n )$$ is zero for $j\neq 1$. By Proposition \ref{proposition_semi_purity_logdeRhamWiss} it suffices to show this for $j>1$. Here $\bar x$ is a geometric point over $x$ which we write as the limit of spectra $x_i$ of finite separable extensions of the residue field of $x$. The last equality holds since by tame cohomology commutes with filtered colimits with affine transition maps (\cite[Thm. 4.7]{HuebnerSchmidt2021}). Since $i^!W_r\Omega_{X,\log,t}^n=i^*(\ker(W_r\Omega_{X,\log,t}^n\to j_*j^* W_r\Omega_{X,\log,t}^n))$ we have that $H^0(\bar x,R^ji^! W_r\Omega_{X,\log,t}^n )=H^j_{\bar x}(\Spec A,W_r\Omega_{X,\log,t}^n )$, where $\Spec A$ is the henselisation of $X\times_k\Spec \overline{k(x)}$ in $\Spec \overline{k(x)}$. In particular, $A$ is strictly henselian.

There is a long exact sequence 
$$H^{j-1}(A,W_r\Omega_{X,\log,t}^n) \to H^{j-1}(K(A),W_r\Omega_{X,\log,t}^n)\to H_{\bar x}^j(A,W_r\Omega_{X,\log,t}^n)\to H^{j}(A,W_r\Omega_{X,\log,t}^n)$$
and since $H^{j}(A,W_r\Omega_{X,\log,t}^n)=0$ for $j\geq 1$, 
$$H^{j-1}(K(A),W_r\Omega_{X,\log,t}^n)\cong  H_{\bar x}^j(A,W_r\Omega_{X,\log,t}^n)$$
for $j\geq 2$. 
In order to show that the former group is zero, we compare it with its analogue in the strongly \'etale topology (for the definition see for example \cite[Sec. 2]{Huebner2023}). The Leray spectral sequence associated to the change of sites
$$\phi:(\Spec(K(A))/k)_t\to \Spec(K(A))_{s\et}$$ induces the boundary map $$H^{j}(\Spec(K(A))_{s\et},W_r\Omega_{X,\log,s\et}^n)\to H^{j}((\Spec(K(A))/k)_t,W_r\Omega_{X,\log,t}^n).$$
In order to show that it is an isomorphism, it suffices to show that $R^i\phi_* W_r\Omega_{X,\log,t}^n=0$ for $i>0$. Let $K(A)^{tr}$ be the maximal tamely ramified extension of $K(A)$. Then $$R^i\phi_* W_r\Omega_{X,\log,t}^n=H^i(\Gal(K(A)^{tr}/K(A)),W_r\Omega_{X,\log,t}^n(K(A)^{tr})).$$ Since $\Gal(K(A)^{tr}/K(A))$ is of degree prime to $p$ and since $W_r\Omega_{X,\log,t}^n$ is a $p$-torsion it follows that the latter group is zero \cite[Prop. 1.6.2]{NSW2008} (see also \cite[Prop. 8.5]{Huebner2021}). Finally, the group $$H^{j}(\Spec(K(A))_{s\et},W_r\Omega_{X,\log,s\et}^n)=0$$ for $j\geq 1$ since $\Spec(K(A))$ does not have a non-trivial cover in the strongly \'etale topology.
\end{proof}

\subsection{The projection formula and strict $\bb A^1$-invariance for a line over a field}
\begin{proposition}
Let $E$ be a vector bundle of rank $r+1$ over $X$ and $\pi:\bb P(E)\to X$ be the associated projective bundle.
Then there is an isomorphism
$$\bigoplus_{i=0}^r R\alpha_*W_r\Omega_{X,\log}^{n-i}[-i])\xto{\cong} R\pi_*R\alpha_*W_r\Omega_{\bb P(E),\log,t}^n$$
in $D(X_t,\bb Z/p^r)$.
\end{proposition}
\begin{proof}
Applying $R\alpha_*$ to the projective bundle formula
\begin{equation}\label{equation_projective bundle formula}
\bigoplus_{i=0}^r W_r\Omega_{X,\log}^{n-i}[-i]\xto{\cong} R\pi_*W_r\Omega_{\bb P(E),\log}^n
\end{equation}
in the \'etale topology (see \cite[Cor. 2.1.12]{Gr85}), we get that
$$\bigoplus_{i=0}^r R\alpha_*W_r\Omega_{X,\log}^{n-i}[-i])\xto{\cong} R\alpha_*R\pi_*W_r\Omega_{\bb P(E),\log}^n\cong R\pi_*R\alpha_*W_r\Omega_{\bb P(E),\log,t}^n.$$
\end{proof}

\begin{proposition}[Strict $\bb A^1$-invariance for a line over a field.]\label{proposition_strict_A1invariance}
Let $k$ be a field of characteristic $p>0$.
Then there is an isomorphism
$$H^i(k,W_r\Omega_{k,\log,t}^n) \xto{\cong} H^i(\bb A_k,W_r\Omega_{\bb A_k,\log,t}^n).$$
\end{proposition}
\begin{proof}
Let $x\in X$ be a closed point with residue field $k$. Consider the localisation sequence
$$\dots \to H^i_x(\bb P_k,W_r\Omega_{\bb P_k,\log,t}^n) \xto{} H^i(\bb P_k,W_r\Omega_{\bb P_k,\log,t}^n) \xto{} H^i(\bb A_k,W_r\Omega_{\bb A_k,\log,t}^n)\to \dots .$$
By Theorem \ref{theorem_tame_coh_proper} and the projective bundle formula (\ref{equation_projective bundle formula}) there is an isomorphism
$$H^{i-1}(k,W_r\Omega_{k,\log,t}^{n-1}) \oplus H^i(k,W_r\Omega_{k,\log,t}^n) \xto{\cong} H^i(\bb P_k,W_r\Omega_{\bb P_k,\log,t}^n)$$
and by purity for curves, Proposition \ref{proposition_purity_curves}, there is
an isomorphism
$$H^{i-1}(k,W_r\Omega_{k,\log,t}^{n-1}) \xto{\cong} H^i_x(\bb P_k,W_r\Omega_{\bb P_k,\log,t}^n).$$
Combining these facts with the localisation sequence implies the proposition.
\end{proof}

\subsection{The Gersten conjecture}
In \cite{GrosSuwa1988} Gros-Suwa prove that for a smooth $d$-dimensional scheme $X$ over a perfect field $k$ of characteristic $p>0$ there is an exact sequence, also called Gersten resolution,
$$0\to R^j\epsilon_*(W_r\Omega_{X,\log,\et}^n) \to \bigoplus_{x\in X^{(0)}}i_{*,x}H^{j}_x(X,W_r\Omega_{X,\log,\et}^n)\to \bigoplus_{x\in X^{(1)}}i_{*,x}H^{j+1}_x(X,W_r\Omega_{X,\log,\et}^n)\to $$
$$\dots\to \bigoplus_{x\in X^{(d)}}i_{*,x}H^{j+d}_x(X,W_r\Omega_{X,\log,\et}^n)\to 0. $$
For $j=0$, using semi-purity for log deRham-Witt sheaves, this gives a resolution of $\epsilon_*(W_r\Omega_{X,\log,\et}^n)$ in terms of log deRham-Witt sheaves of fields.

Combining these remarks with the purity conjecture \ref{conjecture_purity}(ii) for log deRham-Witt sheaves in the tame topology leads to the following conjecture which is, we believe, the optimal Gersten type resolution one can hope for.
\begin{conjecture}\label{conjecture_Gersten_logdRW}
Let $k$ be a perfect field of characteristic $p>0$ and $X$ be a smooth $d$-dimensional $k$-scheme. 
Then there is an exact sequence of sheaves
$$0\to R^j\beta_*(W_r\Omega_{X,\log,t}^n) \to \bigoplus_{x\in X^{(0)}}i_{*,x}H^{j}(x,W_r\Omega_{x,\log,t}^n)\to \bigoplus_{x\in X^{(1)}}i_{*,x}H^{j}(x,W_r\Omega_{x,\log,t}^{n-1})\to $$
$$\dots\to \bigoplus_{x\in X^{(d)}}i_{*,x}H^{j}(x,W_r\Omega_{x,\log,t}^{n-d})\to 0. $$
\end{conjecture}

If $j=0$, then Conjecture \ref{conjecture_Gersten_logdRW} holds because of Proposition \ref{proposition_BL_dRW} and the above mentioned result due to Gros-Suwa.
The following proposition shows Conjecture \ref{conjecture_Gersten_logdRW} in the case of curves. We remark that the same proof goes through in higher dimension assuming that the cohomology of tame logarithmic deRham-Witt sheaves is strictly $\bb A^1$-invariant in the sense of Proposition \ref{proposition_strict_A1invariance}.

\begin{proposition}\label{proposition_Gersten_curve}
Let $X$ be a smooth curve over a field $k$ of characteristic $p>0$. Let $x\in X$ be a closed point and $A=\cal O_{X,x}$. Then there are short exact sequences
$$ 0  \to W_r\Omega_{X,\log,t}^n(A) \to W_r\Omega_{X,\log,t}^n(K(A))\to W_r\Omega_{k(x),\log,t}^{n-1}(k(x))\to 0,$$
and
$$0\to H^1(A,W_r\Omega_{A,\log,t}^n)\to  H^1((K(A),W_r\Omega_{X,\log,t}^n) \to  H^{1}(k(x),W_r\Omega_{k(x),\log,t}^{n-1}) \to 0 $$
and isomorphisms $$H^j(A,W_r\Omega_{X,\log,t}^n)\xto{\cong}  H^j(K(A),W_r\Omega_{X,\log,t}^n)$$
for $j\geq 2$.
\end{proposition}
\begin{proof}
Consider the localisation sequence  
$$\dots \to H^j_x(A,W_r\Omega_{A,\log,t}^n))\to H^j(A,W_r\Omega_{A,\log,t}^n)\to  H^j((K(A),W_r\Omega_{X,\log,t}^n) \to H^{j+1}_x(A,W_r\Omega_{A,\log,t}^n)\to \dots $$
But by purity, i.e. Proposition \ref{proposition_purity_curves} and Corollary \ref{corollary_purity}, we have that 
$$H^j_x(A,W_r\Omega_{A,\log,t}^n))\cong H^{j-1}(k(x),W_r\Omega_{k(x),\log,t}^{n-1}))\cong \begin{cases} k(x) \quad \mathrm{for} \quad n=j=1\\ 0 \quad \mathrm{for} \quad j=0 \quad \mathrm{or} \quad j>2 \;.\\ 0 \quad \mathrm{for} \quad n=0 \quad \mathrm{or} \quad n>1 \end{cases}
$$
This gives us a short exact sequence
$$ 0  \to W_r\Omega_{X,\log,t}^n(A) \to W_r\Omega_{X,\log,t}^n(K(A))\to W_r\Omega_{k(x),\log,t}^{n-1}(k(x))\to 0$$
which coincides with the Gersten resolution for Milnor K-theory:
$$ 0  \to K_n^M(A)/p^r \to K_n^M(K(A))/p^r \to K^M_{n-1}(k(x))/p^r\to 0.$$
The sequence continues as follows:
$$0\to H^1(A,W_r\Omega_{A,\log,t}^n)\to  H^1((K(A),W_r\Omega_{X,\log,t}^n) \to H^{2}_x(A,W_r\Omega_{A,\log,t}^n)\cong H^{1}(k(x),W_r\Omega_{k(x),\log,t}^{n-1})) $$
$$\to H^2(A,W_r\Omega_{A,\log,t}^n)\to  H^2((K(A),W_r\Omega_{X,\log,t}^n) \to H^{3}_x(A,W_r\Omega_{A,\log,t}^n)\cong H^{2}(k(x),W_r\Omega_{k(x),\log,t}^{n-1}))=0 $$
and isomorphisms $$H^j(A,W_r\Omega_{X,\log,t}^n)\xto{\cong}  H^j(K(A),W_r\Omega_{X,\log,t}^n)$$
for $j\geq 3$. 

It remains to show that the map $H^{2}_x(A,W_r\Omega_{A,\log,t}^n)\to H^{2}(A,W_r\Omega_{A,\log,t}^n)$ is zero. In order to show this, we show that for each Zariski neighbourhood $U$ of $x$  we can find a neighbourhood $x\in U'\subset U$ such that the map $H^{2}_x(U',W_r\Omega_{U',\log,t}^n)\to H^{2}(U',W_r\Omega_{U',\log,t}^n)$ is zero. By a standard norm argument we may assume that $k$ is an infinite field. By Gabber's geometric presentation lemma \cite[Thm. 3.1.1]{CHK97} we can find a neighbourhood $x\in U'\subset U$ such that there is  an \'etale map
$$\varphi:U'\to \bb A^1_k$$
which is a Nisnevich neighbourhood of $x$, i.e. the preimage of $\varphi$ consists of one closed point $x'$ and $\varphi$ induces an isomorphism $k(x)\xto{\cong} k(\varphi^{-1}(x))$. Let $s_0:\Spec(k(x))\to \bb P^1_k$ (resp. $s_\infty:\Spec(k(x))\to \bb P^1_k$) be the inclusion at zero (resp. infinity). The statement now follows from the following commutative diagram in which the isomorphisms on the left are due to excision.

$$\xymatrix{
 H^{2}_x(U',W_r\Omega_{U',\log,t}^n) \ar[r]^{}  & H^{2}(U',W_r\Omega_{U',\log,t}^n) & & \\
 H^{2}_x(\bb A^1_k,W_r\Omega_{\bb A^1_k,\log,t}^n) \ar[r] \ar[u]_\cong  & H^{2}(\bb A^1_k,W_r\Omega_{\bb A^1_k,\log,t}^n) \ar[u] & & \\
 &  \ar[u] & H^{2}(k(x) ,W_r\Omega_{k(x),\log,t}^n) \ar[ul]^{} & \\
H^{2}_x(\bb P^1_k,W_r\Omega_{\bb P^1_k,\log,t}^n)\ar[r] \ar[uu]_\cong  & H^{2}(\bb P^1_k,W_r\Omega_{\bb P^1_k,\log,t}^n) \ar[uu]^{} \ar[r]^{} \ar[ur]^{s_\infty^*} & H^{2}(\bb P^1_k-x ,W_r\Omega_{\bb P^1_k,\log,t}^n) \ar[u]^{} 
}$$
Indeed, the lower horizontal composition is zero and the upper triangle commutes since by strict $\bb A^1$-invariance (Proposition \ref{proposition_strict_A1invariance}) $s_0^*=s_\infty^*$.
\end{proof}


\begin{remark}[Kato conjectures]
Let $k$ be a perfect field of characteristic $p>0$ and $X$ be a smooth projective $k$-scheme of dimension $d$. 
Let
$$C_{p^r}^n:= \bigoplus_{x\in X^{(0)}}H^{1}(x,W_r\Omega_{x,\log,t}^n)\to \bigoplus_{x\in X^{(1)}}H^{1}(x,W_r\Omega_{x,\log,t}^{n-1})\to ...\to \bigoplus_{x\in X^{(d)}}H^{1}(x,W_r\Omega_{x,\log,t}^{n-d}). $$
Then Kato conjectures \cite{Ka86} that 
$$H^{a}(C_{p^r}^d)\cong \begin{cases} 0 \quad \quad \mathrm{for} \quad a<d \\ \bb Z/p^r \quad \mathrm{for} \quad a=d  \end{cases}
$$
One may also ask if 
$$H^{a}(C_{p^r}^c)\cong \begin{cases} 0 \quad \quad \mathrm{for} \quad a<d \\ \bb Z/p^r \quad \mathrm{for} \quad a=d  \end{cases}
$$ 
for all $c\geq 0$ or at least if these groups are finite. The complexes $C_{p^r}^n$ may be defined as in \cite[p. 150]{Ka86}. We note that a combination of the purity Conjecture \ref{conjecture_purity} combined with Theorem \ref{theorem_tame_coh_proper} allow to define the complexes by taking the global sections of the Gersten complexes of Conjecture \ref{conjecture_Gersten_logdRW}. Contrary to the case with invertible coefficients we get a long exact sequence
$$ \dots\to H^{q-2}(C_{p^r}^n)\to \CH^{n}(X,q,\bb Z/p^r)\to H^{2n-q}(X,W_r\Omega_{\log}^n)\to \dots$$
for all $n$, and not just $n=d$, coming from the exact triangle 
$$ \tau^{>n}R\epsilon_*W_r\Omega_{\log}^n\to \bb Z(n)_X/p^r \to R\epsilon_*W_r\Omega_{X,\log}^n[-n]. $$
\end{remark}

\subsection{Logarithmic deRham-Witt sheaves for SNC schemes}
\begin{definition}\label{definition_nu_X}
Let $X$ be a pure-dimensional scheme of finite type over a perfect field $k$ of positive characteristic $p$.
Let $\tau\in \{t,\et\}$. We define the sheaves $\nu^{n}_{X,r,\tau}$ on $X_{\tau}$ as
$$\nu^{n}_{X,r,\tau}:= \ker \left( \bigoplus_{x\in X^{(0)}}i_{x*}W_r\Omega^{n}_{x,\log,\tau}\to \bigoplus_{x\in X^{(1)}}i_{x*}W_r\Omega^{n-1}_{x,\log,\tau} \right), $$
where $i_{x*}$ is the pushforward in the $\tau$-topology.
\end{definition}

\begin{remark}
There is an inclusion of \'etale sheaves $W_r\Omega^{n}_{X,\log,\et} \subset \nu^{n}_{X,r,\et}$ (\cite[4.2.1]{Sato2007}).
If $X$ is regular, then this inclusion is an equality by \cite{GrosSuwa1988}.
\end{remark}

\begin{proposition}\label{pushforward_nu(n)}
Let $X$ be a pure-dimensional scheme of finite type over a perfect field $k$ of positive characteristic $p$.
Then
$$\tau_{\leq 0}R\alpha_*\nu^{n}_{X,r,\et}\cong \ker \left( \bigoplus_{x\in X^{(0)}}i_{x*}W_r\Omega^{n}_{x,\log,t}\to \bigoplus_{x\in X^{(1)}}i_{x*}W_r\Omega^{n-1}_{x,\log,t} \right)=\nu^{n}_{X,r,t}. $$
\end{proposition}
\begin{proof}
This follows from the facts that $\alpha$ is left exact, that $R\alpha_*$ commutes with $i_*$ by Proposition \ref{proposition_commutation_pushforward} and that $R\alpha_*$ commutes with pseudo-filtered colimits.
\end{proof}

\section{$p$-adic \'etale Tate twists}\label{section_etaletatetwists}

Let $A$ be a discrete valuation ring with perfect residue field of characteristic $p>0$ and fraction field $K$ of characteristic zero. Let $X$ be a regular semistable scheme which is flat of finite type over $S=\Spec A$. Let $Y$ be the special fiber and $X_K$ the generic fiber of $X$ with inclusions
$X_K\xto{j} X \xleftarrow{i} Y.$
In \cite{Sa07}, Sato defines a pair $(\T_r(n)_X,t')$, with $\T_r(n)_X\in D^b(X_{\et},\bb Z/p^r\bb Z)$ using the distinguished triangle of the form
$$i_*\nu^{n-1}_{Y,r}[-n-1]\xto{g} \T_r(n)_X\xto{t'}\tau_{\leq n}Rj_*\mu_{p^r}^{\otimes n}\xto{\sigma_{X,r}(n)}i_*\nu^{n-1}_{Y,r}[-n]$$
in which the map $\sigma_{X,r}(n)$ arises from the sequence of sheaves
$$R^nj_*\mu_{p^r}^{\otimes n}\to \bigoplus_{x\in Y^{(0)}}i_{x*}W_r\Omega^{n-1}_{x,\log}\to \bigoplus_{x\in Y^{(1)}}i_{x*}W_r\Omega^{n-2}_{x,\log}$$
in which the maps are induced by the boundary maps on Galois cohomology defined by Kato in \cite[§1]{Ka86}. By Lemma \ref{lemma_uniqueness_D} the pair $(\T_r(n)_X,t')$ is unique up to unique isomorphism and the complexes $\T_r(n)_X$ are called \textit{$p$-adic \'etale Tate twists}.

\begin{remark}
If $X$ is smooth over $S$, then $\nu^{n-1}_{Y,r}=W_r\Omega^{n-1}_{Y,\log}$. In this case the objects $\T_r(n)_X$ were first defined by Schneider \cite{Schneider1994} and extensively studied by Geisser \cite{Geisser2004}.
\end{remark}

Further, Sato proves that the $\T_r(n)_X\in D^b(X_{\et},\bb Z/p^r\bb Z)$ satisfy the following properties (and can in fact, vice versa, be defined by properties T1-T4):

\noindent{\bf T1 (Trivialization).}
There is an isomorphism
       $t : j^*\T_r(n)_X \simeq \mu_{p^r}^{\otimes n}$.
\par
\medskip
\noindent
{\bf T2 (Acyclicity).}
$\T_r(n)_X$ is concentrated in $[0,n]$, i.e.,
   the $q$-th cohomology sheaf is zero unless $0 \leq q \leq n$.
\par
\medskip
\noindent
{\bf T3 (Purity).}
For a locally closed regular subscheme $i : Z \hookrightarrow X$
          of characteristic $p$ and of codimension
          $c \,(\ge 1)$, there is a Gysin isomorphism
$W_r\Omega_{Z,\log}^{n-c}[-n-c] \xto{\cong}\tau_{\leq n+c} R i^!\T_r(n)_X$ in $D^b(Z_{\et},\bb Z/p^r\bb Z).$
\par
\medskip
\noindent
{\bf T4 (Compatibility).}
Let $i_y: y \hookrightarrow X$ and $i_x: x\hookrightarrow X$ be points on $X$
    with $\mathrm{ch}(x)=p$, $x \in \overline{ \{ y \} }$ and $\mathrm{codim}_X(x)=\mathrm{codim}_X(y)+1$.
Put $c:=\codim_X(x)$. Then the connecting homomorphism
$$
\begin{CD}
    R^{n+c-1}i_{y*}(Ri_y^!\T_r(n)_X) @>>> R^{n+c}i_{x*}(Ri_x^!\T_r(n)_X)
\end{CD}
$$
in localization theory
agrees with the (sheafified) boundary map of
    Galois cohomology groups due to Kato (see \cite[Sec. 1.8]{Sa07}) 
$$
\begin{CD}
    \left. \begin{cases}
     R^{n-c+1}i_{y*}\mu_{p^r}^{\otimes n-c+1} \quad& (\mathrm{ch}(y)=0)\\
     i_{y*} W_r\Omega^{n-c+1}_{y,\log}  \quad& (\mathrm{ch}(y)=p)
    \end{cases}
    \right\}
      @>>>
         i_{x*}W_r\Omega^{n-c}_{x,\log}
\end{CD}
$$
up to a sign depending only on $(\mathrm{ch}(y),c)$, via Gysin isomorphisms.
Here the Gysin map for $i_y$ with $\mathrm{ch}(y)=0$
   is defined by the isomorphism $t$ in {\bf T1} and Deligne's cycle class
     in $R^{2c-2}i_y^!\mu_{p^r}^{\otimes c-1}$.
\par
\medskip
\noindent
{\bf T5 (Product structure).}
{\it There is a unique morphism
$
\T_r(m)_X \otimes^{\bb L} \T_r(n)_X
      \to \T_r(m+n)_X$ in $D^-(X_{\et},\bb Z/p^r\bb Z)$
that extends the natural isomorphism
         $\mu_{p^r}^{\otimes m} \otimes \mu_{p^r}^{\otimes n}
           \simeq \mu_{p^r}^{\otimes m+n}$ on $X[1/p]$.}
\par
\medskip
Let $Z$ be another scheme which is flat of finite type over $S$
    and for which the objects $\T_r(n)_Z$ $(n \ge 0, r \ge 1)$ are defined.
Let $f : Z \to X$ be a morphism of schemes and let
    $\psi : Z[1/p] \to X[1/p]$ be the induced morphism. Then one has the following functoriality properties:
\par
\medskip    
\noindent{\bf T6 (Contravariant functoriality).}
There is a unique morphism
$$
\begin{CD}
f^*\T_r(n)_X
       @>>> \T_r(n)_Z 
\end{CD}
$$
in $D^b(Z_{\et},\bb Z/p^r\bb Z)$ that extends the natural isomorphism
    $\psi^*\mu_{p^r}^{\otimes n} \simeq \mu_{p^r}^{\otimes n}$
     on $Z[1/p]$.
\par
\medskip
\noindent
{\bf T7 (Covariant funtoriality).}
Assume that $f$ is proper, and put
    $c:=\dim(X)-\dim(Z)$. Then there is a unique morphism
$$
\begin{CD}
   Rf_*\T_r(n-c)_Z[-2c]
       @>>> \T_r(n)_X
\end{CD}
$$
in $D^b(X_{\et},\bb Z/p^r\bb Z)$ that extends the trace morphism
   $R\psi_*\mu_{p^r}^{\otimes n-c}[-2c] \to \mu_{p^r}^{\otimes n}$
     on $X[1/p]$.
\smallskip
\noindent

\begin{remark}\label{proposition_purity_tate_twists} Let $X$ be as above and $i : Z \hookrightarrow X$ a regular closed subscheme of characteristic $p$ and of codimension $c \,(\ge 1)$. Then by purity (T3) there are isomorphisms
$$H^{c+n}_Z(X,\T_r(n))\cong H^{n-c}(Z,W_r\Omega^{n-c}_{Z,\log}[-(n-c)]).$$
In particular, for $x\in X$ a point of characteristic $p$,
$$H^{c+n}_x(X,\T_r(n))\cong H^{n-c}(x,W_r\Omega^{n-c}_{x,\log}[-(n-c)])=:  H^{n-c}(x,\bb Z/p^r(n-c)).$$
\end{remark}


\section{$p$-adic tame Tate twists}\label{section_tametatetwists}
In this section let $A$ be a discrete valuation ring with perfect residue field of characteristic $p>0$ and fraction field $K$ of characteristic zero. Let $X$ be a regular semistable scheme which is flat of finite type over $S=\Spec A$. Let $Y$ be the special fiber and $X_K$ the generic fiber with inclusions
$X_K\xto{j} X \xleftarrow{i} Y.$

\subsection{The definition and first properties}
\begin{definition}
We define \textit{$p$-adic tame Tate twists} as follows:
$$\T_r^t(n)_X:=\tau_{\leq n}R\alpha_*\T_r(n)_X\in D^b((X/S)_{t},\bb Z/p^r\bb Z).$$
\end{definition}

Recall that $j$ denotes the morphism of sites $(X_K/K)_t \xrightarrow{} (X/S)_\t$.
\begin{lemma}\label{proposition_pushforward_vanishing_cycles}
There is an isomorphism
$$R\alpha_*Rj_*\mu_{p^r}^{\otimes n}\xto{\cong}Rj_*\alpha_*\mu_{p^r}^{\otimes n}\stackrel{}{=}Rj_*\mu_{p^r}^{\otimes n}$$
in $D^b((X/S)_{t},\bb Z/p^r\bb Z)$.
\end{lemma}
\begin{proof}
By Proposition \ref{proposition_commutation_pushforward}(i) there is an isomorphism of derived functors
$$R\alpha_*Rj_*\mu_{p^r}^{\otimes n}\xto{\cong}Rj_*R\alpha_*\mu_{p^r}^{\otimes n}$$
and by Theorem \ref{theorem_tame_coh_l_coef} there is an isomorphism $R\alpha_*\mu_{p^r}^{\otimes n}\cong \alpha_*\mu_{p^r}^{\otimes n}$ since $p$ is invertible on $X_K/K$.
\end{proof}

\begin{proposition}\label{proposition_distinguished_triangle}
The object $\T_r^t(n)_X$ fits into a distinguished triangle of the form
$$i_*\nu^{n-1}_{Y,r,t}[-n-1]\xto{g} \T^t_r(n)_X\xto{t'}\tau_{\leq n}Rj_*\mu_{p^r}^{\otimes n}\xto{\sigma_{X,r}(n)}i_*\nu^{n-1}_{Y,r,t}[-n]$$
in $D^b((X/S)_{t},\bb Z/p^r\bb Z)$.
\end{proposition}

\begin{proof}
Since derived  functors are exact, we get a distinguished triangle of the form
$$R\alpha_*\T_r(n)_X\to R\alpha_*\tau_{\leq n}Rj_*\mu_{p^r}^{\otimes n}\xto{\sigma_{X,r}(n)}R\alpha_*i_*\nu^{n-1}_{Y,r,\et}[-n]$$
in $D^b((X/S)_{t},\bb Z/p^r\bb Z)$.
Applying the truncation functor $\tau_{\leq n}$, in general, one does not again get a distinguished triangle. However, in this case the triangle
$$\T_r^t(n)_X\to \tau_{\leq n}R\alpha_*\tau_{\leq n}Rj_*\mu_{p^r}^{\otimes n}\xto{\sigma_{X,r}(n)}\tau_{\leq n}R\alpha_*i_*\nu^{n-1}_{Y,r,\et}[-n]$$
is exact. Indeed, by Lemma \ref{lemma_grothendieck_sp_seq}(iii) and Lemma \ref{proposition_pushforward_vanishing_cycles} we have that $$\tau_{\leq n}R\alpha_*\tau_{\leq n}Rj_*\mu_{p^r}^{\otimes n}\cong \tau_{\leq n}R\alpha_*Rj_*\mu_{p^r}^{\otimes n}\cong \tau_{\leq n}Rj_*\mu_{p^r}^{\otimes n}$$ and by Proposition \ref{proposition_commutation_pushforward}(ii) and Proposition \ref{pushforward_nu(n)} $$\tau_{\leq n}R\alpha_*i_*\nu^{n-1}_{Y,r,\et}[-n]\cong i_*\nu^{n-1}_{Y,r,t}[-n].$$ 
Furthermore, the induced map $$\sigma_{X,r}(n): R^nj_*\mu_{p^r}^{\otimes n}\to i_*\nu^{n-1}_{Y,r,t}$$ is surjective in the tame topology; assuming $X$ is smooth, both terms, evaluated on stalks (that is in particular regular henselian local rings), may be identified with Milnor K-theory (see for example \cite[Thm. 1.7, Rem. 1.8]{LuMo20}), where the map is known to be surjective. The semi-stable case can be deduced from the smooth case using combinatorics of the special fiber.
\end{proof}

\begin{proposition}
$\T_r^t(n)_X$ is concentrated in $[0,n]$ and the triple $(g,\T_r^t(n)_X,t')$ is unique up to unique isomorphism and $g$ is determined by the pair $(\T_r^t(n)_X,t')$.
\end{proposition}
\begin{proof}
By Proposition \ref{proposition_distinguished_triangle} there is a distinguished triangle of the form
$$i_*\nu^{n-1}_{Y,r,t}[-n-1]\xto{g} \T^t_r(n)_X\xto{t'}\tau_{\leq n}Rj_*\mu_{p^r}^{\otimes n}\xto{\sigma_{X,r}(n)}i_*\nu^{n-1}_{Y,r,t}[-n]$$
in $D^b((X/S)_{t},\bb Z/p^r\bb Z)$. Since by definition $\T^t_r(n)_X$ is concentrated in $[0,n]$, Lemma \ref{lemma_hom_D} implies that $$\Hom_{D^b((X/S)_{t},\bb Z/p^r\bb Z)}(\T^t_r(n)_X,i_*\nu^{n-1}_{Y,r,t}[-n-1]) =0.$$ 
The uniqueness now follows from Lemma \ref{lemma_uniqueness_D}(iii).
\end{proof}

\subsection{The Beilinson-Lichtenbaum conjecture for tame Tate twists}
The analogue of the Beilinson-Lichtenbaum conjecture for logarithmic de Rham-Witt sheaves in the \'etale topology in mixed characteristic says that there is an isomorphism 
$$\bb  Z(n)_X/p^r\xrightarrow{\cong} \tau_{\leq n}R\epsilon_*\T_r(n)_X.$$
If $X$ is smooth this is shown in \cite{Geisser2004}. 
\begin{proposition}[Beilinson-Lichtenbaum]\label{proposition_BL}
Let $X$ be smooth over $S$. Then there is an isomorphism
$$\bb  Z(n)_X/p^r\xto{\cong}\tau_{\leq n}R\beta_*\T_r^t(n)_X. $$
\end{proposition}
\begin{proof}
The forgetful functor $\alpha_*$ is left exact and preserves injectives, therefore $R\beta_*R\alpha_*=R\epsilon_*$ and in particular
$$\tau_{\leq n}R\beta_*\T_r^t(n)_X \stackrel{Def.}{\cong}\tau_{\leq n}R\beta_*\tau_{\leq n}R\alpha_*\T_r(n)_X \stackrel{Lem.\; \ref{lemma_grothendieck_sp_seq}(iii)}{\cong} \tau_{\leq n}R\epsilon_*\T_r(n)_X\xleftarrow{\cong} \bb  Z(n)_X/p^r.$$
\end{proof}

\subsection{Purity and product structure}
\begin{proposition}(Semi-Purity)\label{proposition_semi_purity_Tate_twists} Let $i : Z \hookrightarrow X$ be a locally closed regular subscheme
          of characteristic $p$ and of codimension
          $c \,(\ge 1)$.
         Then there is a Gysin isomorphism
$$
\begin{CD}
     W_r\Omega_{Z,\log,t}^{n-c}[-n-c] @>{\simeq}>> \tau_{\leq n+c} R i^!\T^t_r(n)_X
      \quad \hbox{ in } \; D^b(Z_{t},\bb Z/p^r\bb Z).
\end{CD}
$$
\end{proposition}
\begin{proof}
By property T3 of the $p$-adic \'etale Tate twists there is an isomorphism $W_r\Omega_{Z,\log,\et}^{n-c}[-n-c]\xto{\cong}\tau_{\leq n+c} R i^!\T_r(n)_X$ in $D^b(Z_{\et},\bb Z/p^r\bb Z)$. 
Applying $\tau_{\leq n+c}R\alpha_*$, we get an isomorphism $$W_r\Omega_{Z,\log,t}^{n-c}[-n-c]=\tau_{\leq n+c}R\alpha_*W_r\Omega_{Z,\log,\et}^{n-c}[-n-c]\xto{\cong}\tau_{\leq n+c}R\alpha_*\tau_{\leq n+c} R i^!\T_r(n)_X$$ in $D^b(Z_{t},\bb Z/p^r\bb Z)$. By Lemma \ref{lemma_grothendieck_sp_seq}(iii) and Proposition \ref{proposition_commutation_pushforward}(iii) we get the isomorphism of the proposition.
\end{proof}

\begin{remark} Let $i : Z \hookrightarrow X$ be a locally closed regular subscheme
          of characteristic $p$ and of codimension
          $c \,(\ge 1)$.
          We do not know if one should expect the Gysin morphism
$$
\begin{CD}
     W_r\Omega_{Z,\log,t}^{n-c}[-n-c] \to R i^!\T^t_r(n)_X
      \quad \hbox{ in } \; D^b(Z_{t},\bb Z/p^r\bb Z).
\end{CD}
$$
to be an isomorphism without truncation.
\end{remark}

\begin{proposition}(Product structure)\label{proposition_product_structure} There is a morphism
$$\T_r^t(n)_X\otimes^{\bb L}\T_r^t(m)_X\to \T_r^t(n+m)_X  $$
in $D^b((X/S)_{t},\bb Z/p^r\bb Z)$ that extends the natural isomorphism $\mu_{p^r}^{\otimes n}\otimes \mu_{p^r}^{\otimes m}\cong \mu_{p^r}^{\otimes n+m}$ on $X[\frac{1}{p}]$.
\end{proposition}
\begin{proof}
If $n=0$ (or similarly $m=0$), then we get an isomorphism $\bb Z/p^r \bb Z_X\otimes^{\bb L}\T_r^t(m)_X\xto{\cong} \T_r^t(m)_X$. Assume that $m,n\geq 1$. Then there is a composite morphism
$$\T_r^t(n)_X\otimes^{\bb L}\T_r^t(m)_X\to \tau_{\leq n}Rj_*\mu_{p^r}^{\otimes n}\otimes^{\bb L}\tau_{\leq m}Rj_*\mu_{p^r}^{\otimes m}\to \tau_{\leq n+m}Rj_*\mu_{p^r}^{\otimes n+m} \to i_*\nu^{n+m-1}_{Y,r}[-n-m]  $$
The second map is induced by the map $\mu_{p^r}^{\otimes n}\otimes \mu_{p^r}^{\otimes m}\to \mu_{p^r}^{\otimes n+m}$ on the generic fiber. The composition is zero. Indeed, we can check this on the top cohomology sheaves (cf. Lemma \ref{lemma_hom_D}) and by Proposition \ref{proposition_BL} we may refer to the proof of \cite[Prop. 4.2.6]{Sa07}. The uniqueness assertion now follows \ref{lemma_uniqueness_D}(i).
\end{proof}

\subsection{Functoriality}
Let $Z$ be another scheme which is flat of finite type over $S$
    and for which the objects $\T_r(n)_Z$ $(n \ge 0, r \ge 1)$ are defined.
Let $f : Z \to X$ be a morphism of schemes and let
    $\psi : Z[1/p] \to X[1/p]$ be the induced morphism. 
\begin{proposition}
{(Contravariant functoriality).}
There is a unique morphism
$$
\begin{CD}
f^*\T^t_r(n)_X
       @>>> \T^t_r(n)_Z \quad \hbox { in } \; D^b((Z/S)_{\t},\bb Z/p^r\bb Z)
\end{CD}
$$
that extends the natural isomorphism
    $\psi^*\mu_{p^r}^{\otimes n} \simeq \mu_{p^r}^{\otimes n}$
     on $Z[1/p]$.
\end{proposition}
\begin{proof}
The proof is analogous to the proof of Proposition \ref{proposition_product_structure}.
\end{proof}
We also expect $p$-adic tame Tate twists to have the following functoriality property, the proof for $p$-adic \'etale Tate twists is long and difficult \cite[Thm. 7.1.1]{Sa07} and we leave the adaptation for the moment.
\begin{proposition}
{(Covariant funtoriality).}
Assume that $f$ is proper, and put
    $c:=\dim(X)-\dim(Z)$. Then there is a unique morphism
$$
\begin{CD}
   Rf_*\T^t_r(n-c)_Z[-2c]
       @>>> \T^t_r(n)_X \quad \hbox { in } \; D^b((X/S)_{\t},\bb Z/p^r\bb Z)
\end{CD}
$$
that extends the trace morphism
   $R\psi_*\mu_{p^r}^{\otimes n-c}[-2c] \to \mu_{p^r}^{\otimes n}$
     on $X[1/p]$.
\end{proposition}


\subsection{Some open problems}
We list some open problems for $p$-adic tame Tate twists.
\subsubsection*{The Gersten conjecture}
In this section let $X$ be a smooth scheme over $S$ of dimension $d$. 
\begin{conjecture}\label{conjecture_Gersten_tame_Tate_twists}
The complex of sheaves
$$0\to R^n\epsilon_*\cal T^t_r(m) \to \bigoplus_{x\in X^{(0)}}i_{*,x}H^{q}(k(x),\cal T^t_r(m))\to \bigoplus_{x\in X^{(1)}}i_{*,x}H^{q-1}(k(x),\cal T^t_r(m-1))\to $$
$$...\to \bigoplus_{x\in X^{(d)}}i_{*,x}H^{q-d}(x,\cal T^t_r(m-d))\to 0 $$
is exact in the Nisnevich topology.
\end{conjecture}

This is the Gersten conjecture for $p$-adic tame Tate twists and is proved for $p$-adic \'etale Tate twists in \cite{LuedersGerstenTatetwists}. By the Beilinson-Lichtenbaum conjecture \ref{proposition_BL} and \cite[Thm. 1.1]{Geisser2004}, Conjecture \ref{conjecture_Gersten_tame_Tate_twists} holds for $n\leq m$.

\begin{question} 
Does the projection map $\pi:\bb A^1_X\to X$ induces a sequence of isomorphisms
$$H^m_t(X/S,\T_r^t(n))\xto{\cong}H^m_t(\bb A^1_X/S,\pi^*\T_r^t(n))\xto{\cong}H^m_t(\bb A^1_X/S,\T_r^t(n))?$$
\end{question}
A positive answer to this question would, by standard techniques, imply the Gersten conjecture.

\subsubsection*{The cohomology in top degree}

We fix the following notation. Let $A$ be a discrete valuation ring with local parameter $\pi$, perfect residue field $k$ of characteristic $p>0$ and fraction field $K$ of characteristic zero. Let $S=\Spec(A)$. Let $X=\Spec(B)$ be a smooth $S$-scheme. Let $I=(\pi)$ define the special fiber and $B^h$ be the henselisation of $B$ along $I$. Let $i:\Spec B/I\to \Spec B^h$ be the inclusion.

\begin{conjecture}\label{theorem_top_cohomology}
Let the notation be as above. Then $R\epsilon_* i^*\T^t_r(n)\in D((\Spec B/I)_{\rm Nis},\bb Z/p^r)$ is concentrated in degree $[0,n+1]$ and there is an isomorphism
$$H^{n+1}(\Spec B^h,\T^t_r(n))\cong H^{1}(\Spec B/I,W_r\Omega^n_{\log,t}).$$
\end{conjecture}

\begin{remark} Conjecture \ref{theorem_top_cohomology} holds in the \'etale topology by \cite{LuedersGerstenTatetwists}.
By Conjecture \ref{conjecture_affine_anal_pbs}, the restriction map
$$H^{s}(\Spec B^h,\T^t_r(n))\xto\cong H^{s}(\Spec B/I,i^*\T^t_r(n))$$
is an isomorphism for all $s$ which would be the first step in proving Conjecture \ref{theorem_top_cohomology}.
\end{remark}

\section{The cycle class map}\label{section_cycle_class}
Let $S$ be the spectrum of a field or a discrete valuation ring and let $X$ be a scheme of finite type over $X$. We denote by 
$$\bb Z(n)_X/p^r$$
the sheafified (in the Zariski topology) version mod $p^r$ of Bloch's cycle complex \cite{Bl86}, studied by Levine in mixed characteristic in \cite{Le01}.

As an immediate application of the Beilinson-Lichtenbaum conjectures, proved for log deRham-Witt sheaves in Proposition \ref{proposition_BL_dRW} and for $p$-adic tame Tate twist in Proposition \ref{proposition_BL}, we obtain cycle class maps with $\bb Z/p^r$-coefficients to tame cohomology if $X$ is smooth over $S$.
In characteristic $p$ the composition
$$\bb Z(n)_X/p^r\xto{\cong}\tau_{\leq n}R\beta_*W_r\Omega^n_{X,\log,t}[-n]\to R\beta_*W_r\Omega^n_{X,\log,t}[-n]$$
induces a cycle class map
$$\CH^n(X,2n-m)/p^r\cong H^m(X,\bb Z(n)_X/p^r)\to H^m((X/S)_t,W_r\Omega^n_{X,\log,t}).$$
In mixed characteristic the composition
$$\bb Z(n)_X/p^r\xto{\cong}\tau_{\leq n}R\beta_*\T_r^t(n)_X\to R\beta_*\T_r^t(n)_X$$
induces a cycle class map
$$\CH^n(X,2n-m)/p^r\cong H^m(X,\bb Z(n)_X/p^r)\to H^m((X/S)_t,\T_r^t(n)_X).$$

\begin{remark}
We hope that a future construction of a morphism from Chow groups with modulus to compactly supported cohomology of $p$-adic tame Tate twists may be useful in studying higher dimensional tame local class field theory (for this see also \cite{gupta2022tame}).
\end{remark}



\bibliographystyle{acm}
\bibliography{Bibliografie}

\begin{thebibliography}{10}

\bibitem{Bl86}
{\sc Bloch, S.}
\newblock Algebraic cycles and higher {$K$}-theory.
\newblock {\em Adv. in Math. 61}, 3 (1986), 267--304.

\bibitem{BlochKato1986}
{\sc Bloch, S., and Kato, K.}
\newblock {$p$}-adic \'etale cohomology.
\newblock {\em Inst. Hautes \'Etudes Sci. Publ. Math.}, 63 (1986), 107--152.

\bibitem{CHK97}
{\sc Colliot-Th\'el\`ene, J.-L., Hoobler, R.~T., and Kahn, B.}
\newblock The {B}loch-{O}gus-{G}abber theorem.
\newblock In {\em Algebraic {$K$}-theory ({T}oronto, {ON}, 1996)}, vol.~16 of
  {\em Fields Inst. Commun.} Amer. Math. Soc., Providence, RI, 1997,
  pp.~31--94.

\bibitem{GabberAffineAnalogue}
{\sc Gabber, O.}
\newblock {{\(K\)}}-theory of {Henselian} local rings and {Henselian} pairs.
\newblock In {\em Algebraic \(K\)-theory, commutative algebra, and algebraic
  geometry. Proceedings of the US-Italy joint seminar, held June 18-24, 1989 at
  Santa Margherita Ligure, Italy, with support from the National Science
  Foundation and Consiglio Nazionale delle Ricerche}. Providence, RI: American
  Mathematical Society, 1992, pp.~59--70.

\bibitem{Geisser2004}
{\sc Geisser, T.}
\newblock Motivic cohomology over {D}edekind rings.
\newblock {\em Math. Z. 248}, 4 (2004), 773--794.

\bibitem{GeisserLevine2000}
{\sc Geisser, T., and Levine, M.}
\newblock The {$K$}-theory of fields in characteristic {$p$}.
\newblock {\em Invent. Math. 139}, 3 (2000), 459--493.

\bibitem{Gr85}
{\sc Gros, M.}
\newblock Classes de {C}hern et classes de cycles en cohomologie de
  {H}odge-{W}itt logarithmique.
\newblock {\em M\'em. Soc. Math. France (N.S.)}, 21 (1985), 87.

\bibitem{GrosSuwa1988}
{\sc Gros, M., and Suwa, N.}
\newblock La conjecture de {Gersten} pour les faisceaux de {Hodge}-{Witt}
  logarithmique. ({The} {Gersten} conjecture for the logarithmic {Hodge}-{Witt}
  sheaves).
\newblock {\em Duke Math. J. 57}, 2 (1988), 615--628.

\bibitem{gupta2022tame}
{\sc Gupta, R., Krishna, A., and Rathore, J.}
\newblock Tame class field theory over local fields, 2022.

\bibitem{Huebner2021}
{\sc H{\"u}bner, K.}
\newblock The adic tame site.
\newblock {\em Doc. Math. 26\/} (2021), 873--945.

\bibitem{Huebner2023}
{\sc H{\"u}bner, K.}
\newblock Tame and strongly {\'e}tale cohomology of curves.
\newblock {\em Isr. J. Math. 253}, 1 (2023), 1--42.

\bibitem{HuebnerSchmidt2021}
{\sc H\"{u}bner, K., and Schmidt, A.}
\newblock The tame site of a scheme.
\newblock {\em Invent. Math. 223}, 2 (2021), 379--443.

\bibitem{Il79}
{\sc Illusie, L.}
\newblock Complexe de de\thinspace {R}ham-{W}itt et cohomologie cristalline.
\newblock {\em Ann. Sci. \'Ecole Norm. Sup. (4) 12}, 4 (1979), 501--661.

\bibitem{Kato1982}
{\sc Kato, K.}
\newblock Galois cohomology of complete discrete valuation fields.
\newblock Algebraic {{\(K\)}}-theory, {Proc}. {Conf}., {Oberwolfach} 1980,
  {Part} {II}, {Lect}. {Notes} {Math}. 967, 215-238 (1982)., 1982.

\bibitem{Ka86}
{\sc Kato, K.}
\newblock A {H}asse principle for two-dimensional global fields.
\newblock {\em J. Reine Angew. Math. 366\/} (1986), 142--183.
\newblock With an appendix by Jean-Louis Colliot-Th\'el\`ene.

\bibitem{Kerz2009}
{\sc Kerz, M.}
\newblock The {G}ersten conjecture for {M}ilnor {$K$}-theory.
\newblock {\em Invent. Math. 175}, 1 (2009), 1--33.

\bibitem{Kerz2010}
{\sc Kerz, M.}
\newblock Milnor {$K$}-theory of local rings with finite residue fields.
\newblock {\em J. Algebraic Geom. 19}, 1 (2010), 173--191.

\bibitem{KerzSchmidt2010}
{\sc Kerz, M., and Schmidt, A.}
\newblock On different notions of tameness in arithmetic geometry.
\newblock {\em Math. Ann. 346}, 3 (2010), 641--668.

\bibitem{Le01}
{\sc Levine, M.}
\newblock Techniques of localization in the theory of algebraic cycles.
\newblock {\em J. Algebraic Geom. 10}, 2 (2001), 299--363.

\bibitem{LuedersGerstenTatetwists}
{\sc L{\"u}ders, M.}
\newblock The {G}ersten conjecture for p-adic {\'e}tale {T}ate twists and the
  $p$-adic cycle class map.
\newblock {\em arxiv.2403.11853\/} (2024).

\bibitem{LuMo20}
{\sc L{\"u}ders, M., and Morrow, M.}
\newblock Milnor {{\(K\)}}-theory of {{\(p\)}}-adic rings.
\newblock {\em J. Reine Angew. Math. 796\/} (2023), 69--116.

\bibitem{Mi80}
{\sc Milne, J.~S.}
\newblock {\em \'Etale cohomology}, vol.~33 of {\em Princeton Mathematical
  Series}.
\newblock Princeton University Press, Princeton, N.J., 1980.

\bibitem{Milne1986}
{\sc Milne, J.~S.}
\newblock Values of zeta functions of varieties over finite fields.
\newblock {\em Am. J. Math. 108\/} (1986), 297--360.

\bibitem{Morrow2019}
{\sc Morrow, M.}
\newblock {{\(K\)}}-theory and logarithmic hodge-witt sheaves of formal schemes
  in characteristic {{\(p\)}}.
\newblock {\em Ann. Sci. {\'E}c. Norm. Sup{\'e}r. (4) 52}, 6 (2019),
  1537--1601.

\bibitem{NSW2008}
{\sc Neukirch, J., Schmidt, A., and Wingberg, K.}
\newblock {\em Cohomology of number fields}, 2nd ed.~ed., vol.~323 of {\em
  Grundlehren Math. Wiss.}
\newblock Berlin: Springer, 2008.

\bibitem{Sato2007}
{\sc Sato, K.}
\newblock Logarithmic {Hodge}-{Witt} sheaves on normal crossing varieties.
\newblock {\em Math. Z. 257}, 4 (2007), 707--743.

\bibitem{Sa07}
{\sc Sato, K.}
\newblock {$p$}-adic \'etale {T}ate twists and arithmetic duality.
\newblock {\em Ann. Sci. \'Ecole Norm. Sup. (4) 40}, 4 (2007), 519--588.
\newblock With an appendix by Kei Hagihara.

\bibitem{Schneider1994}
{\sc Schneider, P.}
\newblock {{\(p\)}}-adic points of motives.
\newblock In {\em Motives. Proceedings of the summer research conference on
  motives, held at the University of Washington, Seattle, WA, USA, July
  20-August 2, 1991}. Providence, RI: American Mathematical Society, 1994,
  pp.~225--249.

\bibitem{stacks-project}
{\sc {The Stacks Project Authors}}.
\newblock \textit{Stacks Project}.
\newblock \url{https://stacks.math.columbia.edu}, 2018.

\end{thebibliography}

\end{document}